\documentclass[12pt]{amsart}

\usepackage{amsmath, amssymb}

\newcommand{\tr}[2]{\textrm{tr}_{#1}{#2}}
\newcommand{\ti}[1]{\tilde{#1}}

\renewcommand{\leq}{\leqslant}
\renewcommand{\geq}{\geqslant}

\newcommand{\be}{\begin{equation}}
\newcommand{\ee}{\end{equation}}

\newcommand{\ol}{\overline}
\newcommand{\ul}{\underline}
\newcommand{\bM}{\overline{M}}

\begin{document}
\newtheorem{claim}{Claim}
\newtheorem{theorem}{Theorem}[section]
\newtheorem{lemma}[theorem]{Lemma}
\newtheorem{corollary}[theorem]{Corollary}
\newtheorem{proposition}[theorem]{Proposition}
\newtheorem{question}{question}[section]
\newtheorem{definition}[theorem]{Definition}
\newtheorem{remark}[theorem]{Remark}

\numberwithin{equation}{section}

\title[Hessian equations]
{On a class of Hessian type equations on Riemannian manifolds}

\author{Heming Jiao}
\address{School of Mathematics and Institute for Advanced Study in Mathematics, Harbin Institute of Technology,
         Harbin, Heilongjiang, 150001, China}
\email{jiao@hit.edu.cn}

\author{Jinxuan Liu}
\address{School of Mathematics, Harbin Institute of Technology,
         Harbin, Heilongjiang, 150001, China}
\email{16b912021@stu.hit.edu.cn}

\begin{abstract}

In this paper, we consider a class of Hessian type equations which include the $(n-1)$ Monge-Amp\`{e}re equation
on Riemannian manifolds.
The \emph{a priori} $C^2$ estimates and the existence of solutions are established.


\end{abstract}

\maketitle

\section{Introduction}

In this paper, we consider the $C^2$ estimates for the Hessian equations of the form
\begin{equation}
\label{cj-16}
\sigma_k (\lambda (\Delta u g - \nabla^2 u + \chi)) = f (x, u, \nabla u)
\end{equation}
on a compact Riemannian manifold $(M, g)$ of dimension $n \geq 2$ with smooth boundary $\partial M \neq \emptyset$,
where $\chi$ is a $(0, 2)$-tensor field, ${\nabla ^2}u$ is the Hessian of $u$, $\lambda (h)$ denotes the eigenvalues of a $(0, 2)$-tensor
field $h$ with respect to the metric $g$, $f$ is a smooth positive function which may depend on $(x, \nabla u) \in T^* M$ and $u$, and
\[
\sigma_{k} (\lambda) = \sum_ {i_{1} < \ldots < i_{k}}
\lambda_{i_{1}} \ldots \lambda_{i_{k}},\ \ k = 1, \ldots, n,
\]
are the elementary symmetric functions.

Equation \eqref{cj-16} falls in the frame of equations considered in \cite{Guan99} or \cite{GJ15} (see \eqref{cj-8})
where $f$ is assumed to be convex with respect to $\nabla u$. Using the idea of \cite{CJ20} where the star-shaped hypersurface
with a class of prescribed curvatures were considered, we are able to
establish the $C^2$ estimates without this condition.

When $k = n$, the equation \eqref{cj-16} becomes
\begin{equation}
\label{cj-15}
\det (\Delta u g - \nabla^2 u + \chi) = f (x, u, \nabla u)
\end{equation}
which is called the $(n-1)$ Monge-Amp\`{e}re equation. Our interest to study equations \eqref{cj-16} and \eqref{cj-15}
is partly from their complex analogues. The complex $(n-1)$ Monge-Amp\`{e}re equation is closely related to the Gauduchon
conjecture (\cite[\S IV.5]{Gauduchon84}) which was solved by Sz\'ekelyhidi-Tosatti-Weinkove \cite{STW17}.
For more results about the complex $(n-1)$ Monge-Amp\`{e}re equation, the reader is referred to
\cite{FWW10,FWW15,Popovici15,Szekelyhidi18,TW17,TW19} and references therein.

When $k=n=2$, equation \eqref{cj-16} is the classic Monge-Amp\`{e}re equation
\[
\det (\nabla^2 u + \chi) = f (x, u, \nabla u)
\]
which plays a crucial role in many geometric problems such as the Minkowski problem (\cite{CY76}, \cite{Nirenberg53},  \cite{Pogorelov52}, \cite{Pogorelov78}) and the Weyl problem (\cite{GL94}, \cite{HZ95}, \cite{Nirenberg53}).

The G{\aa}rding cone $\Gamma_k$ is defined by
\[
{\Gamma _k} = \{ \lambda  \in {\mathbb{R}^n}\left| {{\sigma _j}\left( \lambda  \right) > 0
{\text{  for  }}j = 1, \ldots ,k} \right.\}, k = 1, \ldots, n.
\]
We state some well known results about $\sigma_k$ in $\Gamma_k$:
\begin{equation}
\label{cj-14}
\frac{\partial \sigma_k}{\partial \lambda_i} = \sigma_{k-1; i} (\lambda) > 0, \mbox{ for all } \lambda \in \Gamma_k \mbox{ and }
  1 \leq i \leq n,
\end{equation}
where $\sigma_{k-1; i} (\lambda)$ is the symmetric function with $\lambda_i = 0$,
\begin{equation}
\label{cj-13}
\sigma_k^{1/k} \mbox{ is concave in } \Gamma_k.
\end{equation}
\begin{definition}
A function $u \in C^2 (M)$ is called admissible if $\lambda (\Delta u g - \nabla^2 u + \chi) (x) \in \Gamma_k$
for all $x \in M$.
\end{definition}
We note that \eqref{cj-16} is elliptic with respect to admissible solutions. In this paper,
we suppose there exists an admissible subsolution $\underline{u} \in C^2 (\ol M)$ satisfying
\begin{equation}
\label{cj-12}
\sigma_k (\Delta \underline{u} g - \nabla^2 \underline{u} + \chi) \geq f (x, \underline{u}, \nabla \underline{u})
   \mbox{ in } M.
\end{equation}

Our main results are the following second order estimates. 
\begin{theorem}
\label{cj-thm-3}
Suppose there exists an admissible function $\ul u \in C^2 (\ol M)$.
Let $u \in C^4 (M) \cap C^2 (\ol M)$ be an admissible solution to \eqref{cj-16}
on $\overline{M}$. Then there exists a positive constant $C$ depending on $|u|_{C^1 (\overline{M})}$,
$|\underline{u}|_{C^2 (\overline{M})}$, $|f|_{C^2}$
and $\inf f$ such that
\begin{equation}
\label{cj-18}
\sup_{\overline{M}} |\nabla^2 u| \leq C (1 + \sup_{\partial M} |\nabla^2 u|).
\end{equation}
\end{theorem}
In this paper we consider the Dirichlet boundary condition
\begin{equation}
\label{cj-19}
u = \varphi \mbox{ on } \partial M.
\end{equation}
\begin{theorem}
\label{cj-thm-4}
Suppose there exists an admissible function $\ul u \in C^2 (\ol M)$ with $\underline{u} = \varphi$ on $\partial M$.
Let $u \in C^3 (M) \cap C^2 (\overline{M})$ be an admissible solution to \eqref{cj-16} and \eqref{cj-19} with $u \geq \underline{u}$ on $\overline{M}$.
Then, there exists a positive constant $C$ depending on $|u|_{C^1 (\overline{M})}$, $|\underline{u}|_{C^2 (\overline{M})}$, $|\varphi|_{C^4 (\partial M)}$,
$|f|_{C^1}$ and $\inf f$ such that
\begin{equation}
\label{cj-20}
\sup_{\partial M} |\nabla^2 u| \leq C.
\end{equation}
\end{theorem}
Since the function $f$ may depend on the gradient of $u$, to obtain the gradient estimates, we need the following growth condition
\begin{equation}
\label{cj-21}
\begin{aligned}
p \cdot \nabla_p f^{1/k} (x, z, p) \leq \,& \bar{f} (x, z) (1 + |p|^{\gamma_1})\\
p \cdot \nabla_x f^{1/k} (x, z, p) + |p|^2 f^{1/k}_z (x, z, p) \geq \,& - \bar{f} (x, z) (1 + |p|^{\gamma_2})
\end{aligned}
\end{equation}
when $|p|$ is sufficiently large, where $\gamma_1 < 2$, $\gamma_2 < 4$ are positive constants and $\bar{f}$ is
a positive continuous function of $(x, z) \in \overline{M} \times \mathbb{R}$.
Based on the \emph{a priori} estimates, in particular, Theorem \ref{cj-thm-3} and Theorem \ref{cj-thm-4}, we can prove the existence
and regularity result as follows.
\begin{theorem}
\label{cj-thm-5}
Suppose there exists an admissible function $\underline{u} \in C^2 (\overline{M})$ satisfying
\eqref{cj-12} and $\underline{u} = \varphi$ on $\partial M$.
Assume that $f > 0$ satisfies \eqref{cj-21} and
\begin{equation}
\label{cj-22}
\sup_{(x, z, p) \in T^* M \times \mathbb{R}} \frac{- f_z (x, z, p)}{f (x, z, p)} < \infty.
\end{equation}
Then there exists an admissible solution $u \in C^\infty (\overline{M})$ of \eqref{cj-16} and \eqref{cj-19}.
Moreover, the solution is unique if $f_z \geq 0$.
\end{theorem}
\begin{remark}
The function $\ul u$ does not need to be a subsolution to derive the
second order estimates in Theorem \ref{cj-thm-3} and Theorem \ref{cj-thm-4}. The conditions that
$\ul u$ is a subsolution (satisfying \eqref{cj-12}) and \eqref{cj-22} are used to derive the existence of solutions
to \eqref{cj-16} and \eqref{cj-19} by the continuity method and degree theory as in \cite{Guan99}.
\end{remark}

The general Hessian type equations as well as their Dirichlet problems and other boundary value problem were studied by many people.
We refer the reader to \cite{ CNS3, CW01, Guan94, Guan99, Guan14, GJ15, GJ16, Ivochkina80, LiYY90, Trudinger95, TW99, Urbas02, W94,
JT18, JT20}
and the reference therein for previous studies.

The equation \eqref{cj-16} is related to the $k$-Hessian equation
\begin{equation}
\label{cj-3}
\sigma_k (\lambda (\nabla^2 u + \chi)) = f (x, u, \nabla u).
\end{equation}
It is an interesting question to establish the second order estimates for the equation \eqref{cj-3} without the
condition that $f$ is convex in $\nabla u$. This is true for the special cases $k=2$, $k=n-1$ and $k=n-2$
according to \cite{GRW15}, \cite{RW19} and \cite{RW02} respectively. It is of interest to consider the general $k$.

\bigskip

{\bf Acknowledgement.}
The first author wish to thank Jianchun Chu for many helpful discussions and comments.
The first author is supported by the National Natural Science Foundation of China (Grant Nos. 11601105, 11871243 and 11671111)
and the Natural Science Foundation of Heilongjiang Province (Grant No. LH2020A002).

\section{Preliminaries}
In this section, we provide some preliminaries which may be used in the following sections.
Throughout the paper let $\nabla$ denote the Levi-Civita connection
of $(M, g)$. The curvature tensor is defined by
\[ R (X, Y) Z = - \nabla_X \nabla_Y Z + \nabla_Y \nabla_X Z
                 + \nabla_{[X, Y]} Z. \]
Under a local frame $e_1, \ldots, e_n$ on $M$ we write
$g_{ij} = g (e_i, e_j)$, $\{g^{ij}\} = \{g_{ij}\}^{-1}$, while
the Christoffel symbols $\Gamma_{ij}^k$ and curvature
coefficients are given respectively by
$\nabla_{e_i} e_j = \Gamma_{ij}^k e_k$ and
\[  R_{ijkl} = g( R (e_k, e_l) e_j, e_i), \;\; R^i_{jkl} = g^{im} R_{mjkl}.  \]
We shall write  $\nabla_i = \nabla_{e_i}$,
$\nabla_{ij} = \nabla_i \nabla_j - \Gamma_{ij}^k \nabla_k $, etc.

For any $v \in C^4 (\ol M)$, we usually
identify $\nabla v$ with its gradient, and use
$\nabla^2 v$ to denote its Hessian which is locally given by
$\nabla_{ij} v = \nabla_i (\nabla_j v) - \Gamma_{ij}^k \nabla_k v$.
We recall that $\nabla_{ij} v =\nabla_{ji} v$ and
\begin{equation}
\label{hess-A70}
 \nabla_{ijk} v - \nabla_{jik} v = R^l_{kij} \nabla_l v,
\end{equation}
\begin{equation}
\label{hess-A80}
\begin{aligned}
\nabla_{ijkl} v - \nabla_{klij} v
= \,& R^m_{ljk} \nabla_{im} v + \nabla_i R^m_{ljk} \nabla_m v
      + R^m_{lik} \nabla_{jm} v \\
  & + R^m_{jik} \nabla_{lm} v
      + R^m_{jil} \nabla_{km} v + \nabla_k R^m_{jil} \nabla_m v.
\end{aligned}
\end{equation}
For $\lambda \in \mathbb{R}^n$ let
\[
\mu_i = \sum_{j \neq i} \lambda_j, i = 1, \ldots, n.
\]
Define the function $h (\lambda)$ on the cone
\[
\Gamma := \{\lambda \in \mathbb{R}^n: (\mu_1, \ldots, \mu_n) \in \Gamma_k\}
\]
by
\begin{equation}
\label{def-h}
h (\lambda) := \sigma_k^{1/k} (\mu_1, \ldots, \mu_n).
\end{equation}
Let
\begin{equation}
\label{def-chi}
\chi_1 := \frac{\tr g (\chi)}{n-1} g - \chi.
\end{equation}
We find the equation \eqref{cj-16} can be rewritten by
\begin{equation}
\label{cj-8}
F (\nabla^2 u + \chi_1) := h (\lambda (\nabla^2 u + \chi_1)) = \tilde{f} (x, u, \nabla u) := f^{1/k} (x, u, \nabla u)
\end{equation}
and a function $u \in C^2 (M)$ is admissible if and only if $\lambda (\nabla^2 u + \chi_1) \in \Gamma$
in $M$.
It is easy to derive from \eqref{cj-14} and \eqref{cj-13} that
\begin{equation}
\label{cj-10}
h_i (\lambda) = \frac{\partial h (\lambda)}{\partial \lambda_i} > 0, \mbox{ in } \Gamma, i = 1, \ldots, n
\end{equation}
and
\begin{equation}
\label{cj-9}
h(\lambda) \mbox{ is concave in } \Gamma.
\end{equation}
Furthermore, $h > 0$ in $\Gamma$ and $h = 0$ on $\partial \Gamma$.
\begin{lemma}
\label{cj-7}
There exists a positive constant $\nu_0$ depending only on $n$ and $k$ such that
\begin{equation}
\label{cj-6}
h_j (\lambda) \geq \nu_0 \big(1 + \sum_i h_i (\lambda)\big), \mbox{ if } \lambda_j < 0, \forall \lambda \in \Gamma.
\end{equation}
\end{lemma}
\begin{proof}
Suppose
\[
\tilde{h}_i (\lambda) = \frac{\partial \sigma_k^{1/k}}{\partial \mu_i}, i = 1, \ldots, n.
\]
It is easy to find
\[
h_i (\lambda) = \sum_{l \neq i} \tilde{h}_l (\lambda), i = 1, \ldots, n
\]
and by Maclaurin's inequality,
\[
\sum_{i=1}^n \tilde{h}_i (\lambda) \geq C_{n, k}
\]
for some positive constant $C_{n, k}$ depending only on $n$ and $k$.
If $\lambda_j < 0$, there exists an index $l \neq j$ such that $\lambda_l > 0$ since $\Gamma \subset \Gamma_1$.
Thus, $\mu_j > \mu_l$ and since $\mu \in \Gamma_k$, we have
\[
\sigma_{k-1; j} (\mu) \leq \sigma_{k-1; l} (\mu).
\]
Hence
$\tilde{h}_j (\lambda) = \frac{1}{k}\sigma_k (\mu)^{1/k - 1}\sigma_{k-1; \mu_j} (\mu) \leq
\frac{1}{k}\sigma_k (\mu)^{1/k - 1}\sigma_{k-1; \mu_l} (\mu) = \tilde{h}_l (\lambda)$. It follows that
\[
\begin{aligned}
2 h_j (\lambda) = 2 \sum_{s \neq j} \tilde{h}_s (\lambda) \geq \,& \sum_{s \neq j} \tilde{h}_s (\lambda)
  + \tilde{h}_l (\lambda) \\
  \geq \,& \sum_{s \neq j} \tilde{h}_s (\lambda) + \tilde{h}_j (\lambda) = \sum_{s=1}^n \tilde{h}_s (\lambda).
\end{aligned}
\]
Therefore,
\[
h_j (\lambda) \geq \frac{1}{2} \sum_{s=1}^n \tilde{h}_s (\lambda)
   \geq \frac{1}{4 (n-1)} \sum_{i=1}^n h_i (\lambda) + \frac{C_{n,k}}{4}
\]
and \eqref{cj-6} is proved.
\end{proof}
It is also easy to see
\begin{equation}
\label{cj-5}
\sum h_i (\lambda)\lambda_i = \sum \tilde{h}_i (\lambda) \mu_i = h (\lambda) > 0 \mbox{ in } \Gamma.
\end{equation}
\begin{lemma}
\label{cj-1}
For any constant $A > 0$ and any compact set $K$ in $\Gamma$ there is a number $R = R (A, K)$ such that
\begin{equation}
\label{cj-2}
h (\lambda_1, \ldots, \lambda_{n-1}, \lambda_n + R) \geq A, \mbox{ for all } \lambda \in K.
\end{equation}
\end{lemma}
\begin{proof}
First there exists a compact set $K' \subset \Gamma_k$ such that
\[\mu = (\sum_{j \neq 1} \lambda_j, \ldots, \sum_{j \neq n} \lambda_j) \in K' \mbox{ for all } \lambda \in K.
\]
Next, for $R > 0$ we have
\[
\begin{aligned}
h^k (\lambda_1, \ldots, \lambda_{n-1}, \lambda_n + R) = \,& \sigma_k (\mu_1 + R, \ldots, \mu_{n-1} + R, \mu_n)\\
 \geq \,& \sigma_k (\mu_1 + R, \mu_2 \ldots, \mu_n)\\
  = \,& (\mu_1 + R) \sigma_{k-1} (\mu_2, \ldots, \mu_n) + \sigma_k (\mu_2, \ldots, \mu_n)\\
 = \,& \sigma_k (\mu) + R \sigma_{k-1} (\mu_2, \ldots, \mu_n) \geq R a_0
\end{aligned}
\]
for some positive constant $a_0$ depending on $K'$. \eqref{cj-2} follows when $R$ is sufficiently large.
\end{proof}
Next, we have (cf. e.g. \cite{W94} and \cite{HMW10}),
\[
\prod_{i=1}^n \sigma_{k-1; i} (\mu) \geq \frac{k^n}{n^n} (C_n^k)^{n/k} [\sigma_k (\mu)]^{n(k-1)/k}
  \mbox{ for any } \mu \in \Gamma_k.
\]
Therefore, for any positive constants $f_1, f_2$ with $0 < \tilde{f}_1 < \tilde{f}_2 < \infty$ there
exists a positive constant $c_0$ depending only on $n$, $k$, $\tilde{f}_1$ and $\tilde{f}_2$ such that
\begin{equation}
\label{cj-4}
\prod_{i=1}^n h_i (\lambda) \geq \prod_{i=1}^n \tilde{h}_i (\lambda) \geq c_0
\end{equation}
for any $\lambda \in \Gamma_{\tilde{f}_1, \tilde{f}_2} := \{\lambda \in \Gamma: \tilde{f}_1 \leq h (\lambda) \leq \tilde{f}_2\}$.

In the following sections, $u$ will denote an admissible solution of \eqref{cj-16} and \eqref{cj-19} with
$u \geq \ul u$ in $M$. For simplicity we denote $U := \nabla^2 u + \chi_1$
and, under a local frame $e_1, \ldots, e_n$,
\[ U_{ij} := U (e_i, e_j) = \nabla_{ij} u + (\chi_1)_{ij}. \]
Denote $G (X) =
\sigma_k^{1/k} (\lambda (X))$ for a $(0, 2)$-tensor field
$X$ on $\overline{M}$. Let
$\eta = \Delta u g - \nabla^2 u + \chi$ and locally
$\eta_{ij} = (g^{pq} \nabla_{pq} u) \delta_{ij} - g^{il} \nabla_{lj} u + g^{il}\chi_{lj}$.
Throughout this paper we denote
\[
G^{ij} = \frac{\partial G}{\partial \eta_{ij}} \mbox{ and } F^{ij} = \frac{\partial F}{\partial U_{ij}}.
\]
Then we have
\begin{equation}
\label{Fij}
F^{ij} = \sum_l G^{ll} g^{ij} - g^{il} G^{lj}.
\end{equation}

\section{A remark for the \emph{a priori} $C^1$ estimates}

We first note that $\Gamma_k \subset \Gamma_1$. Let $u \in C^2 (M) \cap C^0 (\bM)$ be an admissible solution of \eqref{cj-16} and
\eqref{cj-19} with $u \geq \ul u$, we have
\[
\Delta u + \tr g (\chi) > 0 \mbox{ in } M.
\]
Let $\phi$ be the solution of
\[
\Delta \phi + \tr g (\chi) = 0 \mbox{ in } M
\]
with $\phi = \varphi$ on $\partial M$. By the maximum principle, we have
\[
\ul u \leq u \leq \phi.
\]
Therefore, we obtain
\begin{equation}
\label{cj-11}
\sup_{\bM} |u| + \sup_{\partial M} |\nabla u| \leq C,
\end{equation}
where the positive constant $C$ depends only on $|\ul u|_{C^1 (\bM)}$ and $|\phi|_{C^1 (\bM)}$.

\begin{theorem}
\label{gradient}
Let $u \in C^3 (M) \cap C^1 (\ol M)$ be an admissible solution of \eqref{cj-16}. Suppose $f$ satisfies \eqref{cj-21}. Then we have
\begin{equation}
\label{gradient-1}
\sup_{\ol M} |\nabla u| \leq C (1 + \sup_{\partial M} |\nabla u|).
\end{equation}
\end{theorem}
Since $h$ satisfies \eqref{cj-10}-\eqref{cj-5}, Theorem \ref{gradient} follows by
Theorem 5.1 in \cite{Guan99} or Theorem 4.3 in \cite{GJ16}. 
Combining \eqref{cj-11} and \eqref{gradient-1},
the $C^1$ estimates are proved.

\section{Second order estimates}

In this section, we deal with the second order estimates for the Hessian type equation \eqref{cj-16}. The idea is
mainly from \cite{CJ20}. First since
$\lambda (\Delta \underline{u} g - \nabla^2 \underline{u} + \chi) \in \Gamma_k$, we can find a positive constant $\epsilon_0$
such that $\lambda (\Delta \underline{u} g - \nabla^2 \underline{u} + \chi - \epsilon_0 g) (x) \in \Gamma_k$ for all $x \in \overline{M}$.
Note that if $|u|_{C^1 (\ol M)}$ is under control there exist uniform
constants $\ti{f}_2 \geq \ti{f}_1 > 0$ such that
\begin{equation}
\label{fbound}
\ti{f}_1 \leq \ti{f} (x, u, \nabla u) \leq \ti{f}_2  \mbox{ on } \ol M.
\end{equation}
Hence by the
concavity of $\sigma_k^{1/k}$ in $\Gamma_k$, we have
\[
\begin{aligned}
& F^{ij} \nabla_{ij} (\underline{u} - u) - \epsilon_0 \sum G^{ii}\\
= \,& G^{ij} \big[(\Delta \underline{u} - \Delta u) \delta_{ij} - g^{il} (\nabla_{lj} \underline{u} - \nabla_{lj} u) - \epsilon_0 \delta_{ij}\big]\\
  \geq \,& G (\Delta \underline{u} g - \nabla^2 \underline{u} + \chi - \epsilon_0 g) - G (\Delta u g - \nabla^2 u + \chi)\\
  = \,& G (\Delta \underline{u} g - \nabla^2 \underline{u} + \chi - \epsilon_0 g) - \ti{f} (x, u, \nabla u) \geq - C
\end{aligned}
\]
on $\overline{M}$. It follows that
\begin{equation}
\label{cj-23}
F^{ij} \nabla_{ij} (\underline{u} - u) \geq \epsilon_0 \sum G^{ii} - C \ \ \mbox{ on } \overline{M}.
\end{equation}
We are ready to prove Theorem \ref{cj-thm-3}.
\begin{proof}[Proof of Theorem \ref{cj-thm-3}]
Consider the quantity
\[
W = \max_{x \in \overline{M}} \max_{\xi \in T_x M, |\xi| = 1} U_{\xi \xi} e^\phi,
\]
where $U_{\xi\xi} = U (\xi, \xi)$,
\[
\phi = \frac{\delta |\nabla u|^2}{2} + b (\underline{u} - u)
\]
and $\delta$ (sufficiently small), $b$ (sufficiently large) are positive constants to be determined. Suppose $W$
is achieved at an interior point $x_0 \in M$ and $\xi_0 \in T_{x_0} M$. Choose a smooth orthonormal local frame
$e_1, \ldots, e_n$ about $x_0$ such that at $x_0$, $e_1 = \xi_0$, $\nabla_{e_i} e_j = 0$ and $\{U_{ij}\}$
is diagonal. So $\{\eta_{ij} (x_0)\}$, $\{G^{ij} (x_0)\}$ and $\{F^{ij} (x_0)\}$ are all diagonal. We may assume
\[
U_{11} (x_0) \geq \cdots \geq U_{nn} (x_0)
\]
so that
\[
\eta_{11} (x_0) \leq \cdots \leq \eta_{nn} (x_0)
\]
since $\eta_{jj} (x_0) = \sum_{i \neq j} U_{ii} (x_0)$ for $1 \leq j \leq n$. We have, at $x_0$ where the
function $\log U _{11} + \phi$ attains its maximum,
\begin{equation}
\label{cj-24}
\frac{\nabla_{i} U_{11}}{U_{11}} + \delta \nabla_i u \nabla_{ii} u + b \nabla_i (\underline{u} - u) = 0,
\end{equation}
for each $i = 1, \ldots, n$ and
\begin{equation}
\label{cj-25}
F^{ii} \Big\{\frac{\nabla_{ii} U_{11}}{U_{11}} - \Big(\frac{\nabla_{i} U_{11}}{U_{11}}\Big)^2 + \delta (\nabla_{ii} u)^2+ b \nabla_{ii} (\underline{u} - u) + \delta \nabla_l u \nabla_{iil} u\Big\} \leq 0.
\end{equation}
Differentiating the equation \eqref{cj-16} twice we get, at $x_0$,
\begin{equation}
\label{cj-26}
F^{ii} \nabla_{l} U_{ii} = \nabla_{l}' \ti{f} + \ti{f}_u \nabla_l u + \ti{f}_{p_l} \nabla_{ll} u,
\end{equation}
where $\nabla_{l}' \ti{f}$ denotes the partial covariant derivative of $\tilde{f}$ when viewed as
depending on $x \in M$ only
for $l = 1, \ldots, n$. Combining with \eqref{hess-A70} and \eqref{cj-24}, we obtain
\begin{equation}
\label{cj-27}
\begin{aligned}
F^{ii} \nabla_{11} U_{ii} + & G^{ij,st} \nabla_1 \eta_{ij} \nabla_1 \eta_{st} = \nabla_{11} \ti{f}\\
   \geq \,& - C U_{11}^2 + \ti{f}_{p_l} \nabla_{11l} u \geq - C U_{11}^2 - C b U_{11}
\end{aligned}
\end{equation}
provided $U_{11}$ is sufficiently large. Now we consider two cases.

\noindent
{\bf Case 1.} \ $|U_{ii}| \leq \epsilon U_{11}$ for all $i>1$, where the positive constant $\epsilon$ will be determined later.

Thus, we find
\[
|\eta_{11}| \leq (n - 1) \epsilon U_{11}
\]
and
\[
(1 - (n - 2)\epsilon) U_{11} \leq \eta_{22} \leq \cdots \leq \eta_{nn} \leq (1 + (n - 2)\epsilon) U_{11}.
\]
It follows that
\[
\begin{aligned}
\sigma_{k - 1} (\eta) \geq \,& \big[(1 - (n - 2)\epsilon)^{k - 1} - (n - 1) \epsilon (1 + (n - 2)\epsilon)^{k - 2}\big] U_{11}^{k - 1}\\
   \geq \,& \frac{1}{2} U_{11}^{k - 1}
\end{aligned}
\]
if $\epsilon$ is sufficiently small. Therefore, there exists a positive constant $c_0$ depending only on $n$, $k$ and $\tilde{f}_1$ such that
\begin{equation}
\label{cj-30}
\sum G^{ii} = \frac{n - k + 1}{k} f^{1/k - 1} \sigma_{k - 1} (\eta) \geq c_0 U_{11}^{k - 1}.
\end{equation}
As in \cite{Guan14}, using an
inequality due to Andrews~\cite{Andrews94}
and Gerhardt~\cite{Gerhardt96} and by \eqref{hess-A70} and \eqref{cj-24}, we have
\begin{equation}
\label{cj-31}
\begin{aligned}
 - G^{ij, st} & \nabla_1 \eta_{ij} \nabla_1 \eta_{st}
   \geq \sum_{i \neq j} \frac{G^{ii} - G^{jj}}{\eta_{jj} - \eta_{ii}} (\nabla_1 \eta_{ij})^2\\
   \geq \,& 2 \sum_{i \geq 2} \frac{F^{ii} - F^{11}}{U_{11} - U_{ii}} (\nabla_{i} U_{11})^2\\
   \geq \,& \frac{2 (1 - \theta)}{1 + \theta} \sum_{i \geq 2} \frac{F^{ii} (\nabla_{i} U_{11})^2}{U_{11}}
           - \frac{C}{\theta U_{11}} \sum F^{ii}
             - U_{11} F^{11} |\nabla \phi|^2\\
   \geq \,& \sum_{1 \leq i \leq n} \frac{F^{ii} (\nabla_{i} U_{11})^2}{U_{11}} - \frac{C}{\theta U_{11}} \sum F^{ii}\\
             & - C U_{11} (b^2 + \delta^2 U_{11}^2) F^{11}
\end{aligned}
\end{equation}
by fixing $\theta$ small enough so that
\[
\frac{2 (1 - \theta)}{1 + \theta} \geq 1.
\]
Now we can derive from
\eqref{cj-23} and \eqref{cj-25}-\eqref{cj-31} that, at $x_0$,
\begin{equation}
\label{cj-29}
\begin{aligned}
0 \geq \,& - C \sum F^{ii} - C U_{11} - \frac{C}{\theta U_{11}^2} \sum F^{ii} + \delta F^{ii} U_{ii}^2\\
 & - C (b^2 + \delta^2 U_{11}^2) F^{11} - C b
   + b \epsilon_0 \sum G^{ii} - C \delta U_{11}\\
  \geq \,& (\delta - C \delta^2) F^{ii} U_{ii}^2 - C b^2 F^{11} + \frac{b \epsilon_0 c_0}{2} U_{11}^{k - 1}\\
  & - C U_{11} + \frac{b \epsilon_0}{2} \sum G^{ii} - C (n - 1) \sum G^{ii}
\end{aligned}
\end{equation}
provided $U_{11} \geq b$. Suppose $b$ is large enough such that in \eqref{cj-29},
\[
\frac{b \epsilon_0 c_0}{2} U_{11}^{k - 1}
   - C U_{11} + \frac{b \epsilon_0}{2} \sum G^{ii} - C (n - 1) \sum G^{ii} \geq 0.
\]
Therefore, we have
\begin{equation}
\label{cj-34}
0 \geq (\delta - C \delta^2) F^{ii} U_{ii}^2 - C b^2 F^{11}.
\end{equation}

\bigskip

\noindent
{\bf Case 2.} \ $U_{22} > \epsilon U_{11}$ or $U_{nn} < - \epsilon U_{11}$.
\bigskip

By the definition of $F^{ij}$, we have, at $x_0$,
\[
F^{jj} = \sum_{i \neq j} G^{ii} \geq G^{11} \geq \frac{1}{n} \sum G^{ii}, \mbox{ for all } j = 2, \ldots, n.
\]
It follows that, in this case,
\begin{equation}
\label{cj-32}
F^{ii} U_{ii}^2 \geq \frac{\epsilon U_{11}^2}{n} \sum G^{ii}.
\end{equation}
Next, by MacLaurin's inequality, $\sum G^{ii} \geq c_1 > 0$ for some positive constant $c_1$ depending only on $n$ and $k$.
We have, by \eqref{cj-23}-\eqref{cj-27}, \eqref{cj-32} and the concavity of $G$,
\begin{equation}
\label{cj-33}
\begin{aligned}
0 \geq \,& - C U_{11} - C b - F^{ii} (\nabla_i \phi)^2 + \delta F^{ii} U_{ii}^2 - C \sum F^{ii}\\
 \geq \,& \frac{(\delta - C \delta^2)\epsilon}{n} U_{11}^2 \Big(\frac{1}{2} \sum G^{ii} + \frac{c_1}{2}\Big)\\
    & - C U_{11} - C b^2 \sum G^{ii}
\end{aligned}
\end{equation}
provided $U_{11} \geq b$.

Now in view of \eqref{cj-34} and \eqref{cj-33} in both cases, we can choose $0 < \delta \ll 1 \ll b$ to obtain
\[
U_{11} (x_0) \leq \frac{C b}{\epsilon (\delta - C \delta^2)}
\]
and Theorem \ref{cj-thm-3} is proved.
\end{proof}

\section{Boundary estimates for second order derivatives}

The key step to prove Theorem \ref{cj-thm-4} is constructions of suitable barrier functions.
For a point $x_0$ on $\partial M$, we choose
smooth orthonormal local frames $e_1, \ldots, e_n$ around $x_0$ such that
when restricted to $\partial M$, $e_n$ is the interior normal to $\partial M$.

The estimate for $|\nabla_{\alpha \beta} u (x_0)|$ follows naturally from the formula
\begin{equation}
\label{hess-a200}
\nabla_{\alpha \beta} (u - \ul{u})
 = -  \nabla_n (u - \ul{u}) \varPi (e_{\alpha}, e_{\beta}), \;\;
\forall \; 1 \leq \alpha, \beta < n \;\;
\mbox{on  $\partial M$}
\end{equation}
where 
$\varPi$ denotes the second fundamental form of $\partial M$.

For $x \in \overline{M}$ let $\rho (x)$ and $d (x)$ denote the distances
from $x$ to $x_0$ and $\partial M$, respectively,
\[ \rho (x) := \mbox{dist}_{M} (x, x_0), \;\;
      d (x) := \mbox{dist}_{M} (x, \partial M) \]
and $M_{\delta} = \{x \in M : \rho (x) < \delta \}$.
As in \cite{Guan14},
we shall construct barriers of the form
\begin{equation}
\label{cj-35}
 \varPsi
   = A_1 v + A_2 \rho^2 - A_3 \sum_{\beta < n} |\nabla_{\beta} (u - \varphi)|^2,
\end{equation}
where $v = (u - \underline{u}) + t d - \frac{N d^2}{2}$ and $t$, $N$, $A_1$, $A_2$, $A_3$ are positive constants to be chosen.
Define the linear operator $\mathcal{L}$ locally by
\[
\mathcal{L} w := F^{ij} \nabla_{ij} w
    - f^{1/k}_{p_l} (x, u, \nabla u) \nabla_l w, \;\; w \in C^2 (M).
\]

It follows that Lemma 6.2 of \cite{Guan99} also holds for equation \eqref{cj-8} and \eqref{cj-19}.
Thus, there exists positive constants $t, \delta \ll 1 \ll N$ and $\theta_0$ such that
\begin{equation}
\label{cj-36}
\mathcal{L} v \leq - \theta_0 \Big(1 + \sum F^{ii}\Big) \mbox{ and } v \geq 0 \mbox{ in } M_\delta.
\end{equation}
\begin{lemma}
\label{cj-lem3}
Let $h \in C
(\overline{M}_{\delta})$ satisfy $h \leq C \rho^2$ on $\overline{M}_{\delta}
\cap \partial M$ and $h \leq C$ on $\overline{M}_{\delta}$. Then for
any positive constant $K$ there exist uniform positive constants
$A_1 \gg A_2 \gg A_3 \gg 1$ such that $\varPsi \geq h$ on $\partial
M_{\delta}$ and
\begin{equation}
\label{eq3-5}
\mathcal{L} \varPsi \leq - K
\Big(1 + \sum f_i |\lambda_i|
  + \sum F^{ii}\Big)  \;\; \mbox{in $M_{\delta}$},
\end{equation}
where $f_1, \ldots, f_n$ and $\lambda_1, \ldots, \lambda_n$ are eigenvalues of
$\{F^{ij}\}$ and $\{U_{ij}\}$ respectively.
\end{lemma}
We refer the reader to Lemma 5.2 in \cite{GJ15} for the proof of Lemma \ref{cj-lem3}
and a bound of $|\nabla_{\alpha n} u (x_0)|$ ($1 \leq \alpha \leq n-1$) can be obtained
as in \cite{GJ15}. A bound of $|\nabla_{nn} u (x_0)|$ can be obtained as in \cite{Guan99}
since the equation \eqref{cj-8} satisfies \eqref{cj-2} so that \eqref{cj-8} is a special case
of equations considered in \cite{Guan99}. In this paper, we provide a different proof
which depends on the special structure of \eqref{cj-8}. It suffices to prove
\begin{equation}
\label{purenormal}
\nabla_{nn} u (x_0) \leq C
\end{equation}
since the trace of $U$ with respect to the metric $g$, $\mathrm{tr}_g (U) > 0$.

We consider two
cases $k < n$ and $k = n$. First we may assume $\{U_{\alpha \beta}(x_{0})\}_{1\leq\alpha,\beta\leq n-1}$ is diagonal by
rotating $e_1, \ldots, e_{n-1}$. Let $\mu_\alpha = U_{nn} (x_0) + \sum_{\beta \neq \alpha} U_{\beta \beta} (x_0)$ for $1 \leq \alpha \leq n - 1$,
$\mu' = (\mu_1, \ldots, \mu_{n-1})$
and $\mu_n = \sum_{\alpha < n} U_{\alpha \alpha} (x_0)$. When $k < n$, we find, at $x_0$, the equation \eqref{cj-8} can be written by
\begin{equation}
\label{eqn}
\mu_n \sigma_{k - 1} (\mu') + \sigma_k (\mu') - \sum_{\alpha < n} U_{\alpha n}^2 \sigma_{k-2; \alpha} (\mu') = f.
\end{equation}
We see if $U_{nn} (x_0)$ is sufficiently large, so are $\mu_\alpha$ for $1 \leq \alpha \leq n-1$. It follows that $U_{nn} (x_0)$ being sufficiently large
will contradict the equation \eqref{eqn}. Therefore, \eqref{purenormal} is proved.

For the case $k=n$, we have
\begin{lemma}
\label{cj-lem4}
Suppose $k = n$.
Let $\tilde{g}$, $\tilde{\nabla}$ and $\nu$ be the induced metric, connection on $\partial M$ and
the interior unit normal vector field to $\partial M$ respectively.
Let $\lambda' (\tilde{\nabla} \varphi - \nabla_\nu u \varPi + \tilde{\chi}_1)$ denote the eigenvalues of
$\tilde{\nabla} \varphi - \nabla_\nu u \varPi + \tilde{\chi}_1$ with respect to the metric $\tilde{g}$ on $\partial M$.
Then there exists a uniform constant $c_0 > 0$ such that
\begin{equation}
\label{cj-17}
\mathrm{tr}_{\tilde{g}} (\tilde{\nabla} \varphi - \nabla_\nu u \varPi + \tilde{\chi}_1) \geq c_0 \mbox{ on } \partial M,
\end{equation}
where $\mathrm{tr}_{\tilde{g}} (X')$ denotes the trace of $X'$ with respect to $\tilde{g}$ for a $(0, 2)$-tensor field
$X'$ on $\partial M$.
\end{lemma}
\begin{proof}
First, for any $\xi, \eta \in T_x \partial M$ at $x \in \partial M$,
\[
\nabla_{\xi \eta} u = \tilde{\nabla} \varphi - \nabla_\nu u \varPi (\xi, \eta).
\]
Suppose the minimum of $\psi (x) := \mathrm{tr}_{\tilde{g}} (\tilde{\nabla} \varphi - \nabla_\nu u \varPi + \tilde{\chi}_1)(x)$
is achieved at a point $y_{0}\in\partial\Omega$.
It suffices to show $\psi (y_{0})\geq c_{0}>0$. Choose a local orthonormal frame $e_{1},\cdots,e_{n}$ about $y_{0}$
as before such that $\{U_{\alpha \beta}(y_{0})\}_{1\leq\alpha,\beta\leq n-1}$ is diagonal. It follows that
\[
\psi (x) = \sum_{\alpha < n} U_{\alpha \alpha} (x) \mbox{ for } x \in \partial M \mbox{ near } y_0.
\]
Set
\begin{equation}B_{\alpha\beta}=\langle \nabla_{\alpha}e_{\beta},e_{n}\rangle,\ \ 1\leq\alpha,\beta\leq n-1.
\label{B}\end{equation}
Thus, $B_{\alpha\beta}=\varPi(e_{\alpha},e_{\beta})$ on $\partial M$ near $y_0$.
We have
\begin{equation}
\sum_{\alpha<n} U_{\alpha \alpha} (x) = \psi (x)\geq \psi (y_{0})\mbox{ for all }x\in\partial M \mbox{ near }y_{0}.
\label{bs9}\end{equation}
By \eqref{hess-a200} and \eqref{bs9} we have for $x\in\partial M$ near $y_{0}$,
\begin{equation}
\label{bs10}
\begin{aligned}
 \nabla_{n}(u-\underline{u})(x) \sum_{\alpha<n} B_{\alpha\alpha}(x)
\leq \sum_{\alpha<n} \ul U_{\alpha \alpha} (x) - \psi (y_{0}),
\end{aligned}
\end{equation}
where $\ul U := \nabla^2 \underline{u} + \chi_1$.
Since $\lambda (\ul U) \in \Gamma$, there exists a uniform constant $\tilde{c} > 0$ such that
$\sum_{\alpha<n} \ul U_{\alpha \alpha} (y_0) \geq \tilde{c}$. We may assume
$\psi (y_{0})<\frac{\tilde{c}}{2}$ for otherwise we are done. By \eqref{hess-a200}, we have
\begin{equation}
\label{bs11}
\begin{aligned}
 \nabla_{n}(u-\underline{u})(y_{0}) \sum_{\alpha<n} B_{\alpha\alpha}(y_{0})
  \geq \sum_{\alpha < n} \ul U_{\alpha \alpha} (y_0)) - \psi (y_{0}) \geq \frac{\tilde{c}}{2}.
\end{aligned}
\end{equation}
It follows that
$$\sum_{\alpha<n} B_{\alpha\alpha}\geq c_1$$
in $\overline{M}_{\delta'}$ for some uniform positive constant $c_1$ and $\delta'$.
Thus, by \eqref{bs10},
\begin{equation}\nabla_{n}(u-\underline{u})\leq \varPhi \mbox{ on }\partial M \cap\overline{M}_{\delta'},
\label{bs12}\end{equation}
where
\[
\varPhi (x) := \frac{\sum_{\alpha<n} \ul U_{\alpha \alpha} (x)- \psi(y_{0})}{\sum_{\alpha<n} B_{\alpha\alpha}(x)}
\]
is smooth in $M_{\delta'}$. By \cite{Guan14}, we see
\[
|\mathcal{L} \nabla_n (u - \ul u)| \leq C \Big(1 + \sum f_i |\lambda_i|
  + \sum F^{ii}\Big) \mbox{ in } M_{\delta'}.
\]
Thus, by Lemma \ref{cj-lem4} we can choose $\varPsi$ such that
\[
\begin{aligned}
\mathcal{L} (\nabla_n (u - \ul u) - \varPhi - \varPsi) \geq \,& 0 \mbox{ in } M_{\delta_0}\\
  \nabla_n (u - \ul u) - \varPhi - \varPsi \leq \,& 0 \mbox{ on } \partial M_{\delta_0}
\end{aligned}
\]
for some positive constant $\delta_0 \leq \delta'$.
By the maximum principle, we get a bound
\[
\nabla_{nn} u (y_0) \leq C.
\]
Note that when $k=n$, at $y_0$, the equation \eqref{cj-8} is
\begin{equation}
\label{eqn1}
\mu_n \sigma_{n - 1} (\mu') - \sum_{\alpha < n} U_{\alpha n}^2 \sigma_{n-2; \alpha} (\mu') = f,
\end{equation}
where $\mu_i (1 \leq i \leq n)$ is defined as in \eqref{eqn} with $x_0$ replaced by $y_0$. It follows that
\[
\frac{f}{\mu_n} = \sigma_{n - 1} (\mu') + \frac{\sum_{\alpha < n} U_{\alpha n}^2 \sigma_{n-2; \alpha} (\mu')}{\mu_n} \leq \sigma_{n - 1} (\mu') \leq C
\]
since $\lambda(U)(y_{0})$ lies in an \emph{a priori} bounded subset of $\Gamma$. Therefore,
\[
\mu_n = \sum_{\alpha < n} U_{\alpha \alpha} (y_0) = \psi (y_0) \geq \frac{f}{C} \geq \frac{\tilde{f}_1^k}{C} \geq c_0
\]
for some $c_{0}>0$ by \eqref{fbound}.
\end{proof}
\eqref{purenormal} follows from \eqref{cj-17} and \eqref{eqn1} (with $y_0$ replaced by $x_0$).
Theorem \ref{cj-thm-4} is proved.

After establishing the \emph{a priori} $C^2$ estimates, the equation \eqref{cj-8} becomes uniformly elliptic
with respect to admissible solutions by \eqref{fbound}. The $C^{2, \alpha}$ estimates can be proved using Evans-Krylov theory
since \eqref{cj-8} is concave for admissible solutions.
The higher order estimates are derived from the classic Schauder estimates.
Theorem \ref{cj-thm-5} can be proved by standard arguments using the continuity method and degree
theory for which the reader is referred to Section 8 in \cite{Guan99} and \cite{CNS84} where Monge-Amp\`ere equations are treated for more details.

\end{document}